\documentclass[12pt]{amsart}

\usepackage{amssymb,amsmath,latexsym}
\usepackage[varg]{pxfonts}
\usepackage{amsmath, amsthm, amscd, amsfonts, amssymb, graphicx, color}
\title[A product topology convergence scheme ]{A product topology convergence scheme for finding a function that it's values  are       common fixed points of  a family of   representations}
 \author{   Ebrahim  Soori  }
 \thanks{ \!\!\!\!\!\!\!\! \!\!2010 Mathematics Subject Classification:  \\ E-mail address: sori.ebrahim@sci.ui.ac.ir
  \\Tel: +98 9188521850 (E. Soori)}
\theoremstyle{plain}
\newtheorem{lem}{\textbf{Lemma}}[section]
\newtheorem{thm}[lem]{\textbf{Theorem}}
\newtheorem{co}[lem]{\textbf{Corollary}}

\theoremstyle{definition}

\theoremstyle{definition}
\theoremstyle{remark}

\renewcommand{\sc}{\mathcal{S}}

\newcommand{\ud}{\,\mathrm{d}}
\begin{document}
\begin{large}

\maketitle
%\date{January 1, 2011}
%----------additions
%\dedicatory{To my boss}
%%% ----------------------------------------------------------------------
 \begin{center}
 \begin{normalsize}
 Lorestan University, Department  of Mathematics,  P. O. Box 465, Khoramabad, Lorestan, Iran.
 \end{normalsize}
 \end{center}
\begin{abstract}
\begin{normalsize}
In this paper, using  a family of   representations  of nonexpansive mappings, we introduce     an  algorithm in a product space $E^{I}$ consisting of all functions from a nonempty set $I$ to a Banach space $E$.  Then we prove  the product topology convergence   of  the proposed algorithm    to    an element of $E^{I}$ such that   it's  values are        the common fixed points  of      the  representations  of the family.
\end{normalsize}
\end{abstract}
\begin{normalsize}
   \textbf{keywords}:   Product topology; Fixed point; Nonexpansive mapping; Representation;  mean.
   \end{normalsize}
%%% ----------------------------------------------------------------------
%%% ----------------------------------------------------------------------
%\tableofcontents

\section{ Introduction}
Let $C$ be a nonempty closed and convex subset of a Banach space $E$ and
$E^{*}$ be the dual space of $E$. Let $\langle.,.\rangle$   denote the pairing between $E$ and $E^{*}$. The
normalized duality mapping $J: E \rightarrow E^{*}$
is defined by
\begin{align*}
    J(x)=\{f \in E^{*}: \langle x, f \rangle= \|x\|^{2}=\|f\|^{2} \}
\end{align*}
for all $x \in E.$ In the sequel, we use $j$ to denote the single-valued normalized
duality mapping. Let $U = \{x \in E : \|x\| = 1\}$. $E$ is said to be smooth or said
to   have a G$\hat{a}$teaux differentiable norm if the limit
\begin{align*}
    \lim_{t\rightarrow 0}\frac{\|x+ty\|-\|x\|}{t}
\end{align*}
exists for each $ x, y \in U$. $E$ is said to have a uniformly G$\hat{a}$teaux differentiable
norm if for each $y \in U$, the limit is attained uniformly for all $ x \in U$. $E$ is said
to be uniformly smooth or said to   have a uniformly F$r\acute{e}$chet differentiable
norm if the limit is attained uniformly for $x, y \in U$. It is known that if the
norm of $E$ is uniformly G$\hat{a}$teaux differentiable, then the duality mapping $J$
is single valued and uniformly norm to weak$^{*}$ continuous on each bounded
subset of $E$. A Banach space $E$ is   smooth if the duality mapping $J$ of
$E$ is single valued. We know that if $E$ is smooth, then $J$ is norm to weak-star continuous; for more details, see  \cite{tn}.

Let $C$ be a nonempty closed and convex subset of a Banach space $E$.  A mapping $ T$ of $ C $ into itself is called nonexpansive if $\|Tx - Ty\| \leq \|x - y\|,$ for all $x, y \in C$ and a mapping $f$ is an $\alpha_{i}$-contraction on $E $ if  $ \|f (x) -f (y)\| \leq \alpha_{i} \|x - y\|, \;x, y \in E$  such that  $0 \leq\alpha_{i} < 1$.

In this paper   we introduce the following general   algorithm in the product space $E^{I}$ for finding  an element of $E^{I}$ such that   it's  values are        the common fixed points  of      the  representations    $\sc_{i} =\{T_{t,i}:t\in S\}$ of a semigroup $S$ as nonexpansive  mappings from $C_{i}$ into itself  with respect to   a left regular sequence of means   defined on an appropriate subspace of bounded real-valued
functions of the semigroup. On the other hand, our goal is to prove that  there exists a unique sunny nonexpansive retraction $ P_{i} $  of $ C_{i} $ onto $  \rm{Fix}( \sc_{i})$ and $ x_{i} \in C_{i} $ for each $  i \in I$ such that  the sequence  $\{g_{n}:I\rightarrow E \}$ in $E^{I}$       generated by
 \begin{equation*}
  \left\{
\begin{array}{lr}
g_{n}(i)= z_{n,i},\qquad i \in I, \\
      z_{n,i}=   \epsilon_{n} f_{i}(z_{n,i})+(1-\epsilon_{n})T_{\mu_{n,i}}z_{n,i} \qquad  i \in I,
\end{array} \right.
\end{equation*}
      converges to a function  $g:I\rightarrow E$  in $E^{I}$  defined by   $g(i)=P_{i}x_{i}$ in the product topology on $E^{I}$.

\section{preliminaries}
 Let $S$ be a semigroup. We denote by $B(S)$ the Banach space of all bounded real-valued functions defined on $ S $ with supremum
norm. For each $ s \in S$ and $f\in B(S)$ we define $l_{s}$ and  $r_{s}$ in $B(S)$  by\\ \;\indent$\quad (l_{s}f )(t) = f (st)$ ,\qquad $(r_{s}f )(t) = f (ts)$,\quad $ \;\;( t \in S)$.\\
Let $ X$ be a subspace of $B(S)$ containing 1 and let $ X^{*}$ be its topological dual. An element $\mu $ of $X^{*} $ is said to be a mean on $X$ if $\|\mu\|=\mu(1)=1$. We often write $\mu_{t}(f (t))$ instead of $ \mu(f )$ for $ \mu\in X^{*}$ and $f \in X$. Let $ X $ be left invariant (resp. right
invariant), i.e. $ l_{s}(X) \subset X$ (resp. $ r_{s}(X) \subset X$) for each $ s \in S$. A mean $\mu$ on $X$ is said to be left invariant (resp. right invariant) if $\mu(l_{s}f ) =\mu(f )$ (resp. $\mu(r_{s}f )= \mu(f )$) for each $s \in S $ and $f \in X$. $X $ is said to be left (resp. right) amenable if $X$ has a left (resp. right) invariant mean. $X$ is amenable if $X$ is both left and right amenable. As is well known, $B(S)$ is amenable when $S$ is a
commutative semigroup (see  page 29 of \cite{tn}). A net $\{\mu_{\alpha_{i}}\}$ of means on $X$ is said to be left regular if  $$\displaystyle\lim_{\alpha_{i}} \|l^{*}_{s}\mu_{\alpha_{i}}-\mu_{\alpha_{i}}\|= 0,$$  for
each $s \in S$, where $l_{s}^{*}$ is the adjoint operator of $l_{s}$.

Let $f$ be a  function of semigroup $S$ into a reflexive Banach space $E$ such that the weak closure of $\{f(t):t\in S\}$ is weakly
compact and let $X$ be a subspace of $B(S)$ containing all the functions $t \rightarrow \left<f(t),x^{\ast}\right>$ with  $ x^{\ast} \in E^{\ast}  $. We know from \cite{hi} that for any $ \mu\in X^{\ast}$, there
exists a unique element $ f_{\mu}$ in $E$ such that $\left<f_{\mu},x^{\ast}\right>=\mu_{t}\left<f(t),x^{\ast}\right>$
for all $ x^{\ast}\in E^{\ast} $. we denote such $f_{\mu}$ by $\int\!f(t)\ud \mu(t)$. Moreover, if $ \mu$ is a mean on $X$ then from \cite{ki}, $\int\! f(t)\ud \mu(t) \in \overline{\rm{co}}\,\{f(t):t\in S\}$.

Let $C$ be a nonempty closed and convex subset of $E$. Then, a family   $ \sc=\{T_{s}:s\in S\} $ of mappings from $C$ into itself is said to be a  representation of $S$ as nonexpansive mapping on $C$ into itself if $ \sc$ satisfies the following :\\
(1) $T_{st}x=T_{s}T_{t}x$ for all $s, t \in S$ and $x\in C$;\\
(2) for every $s \in S$ the  mapping $T_{s}: C \rightarrow C$ is nonexpansive.\\
 We denote by Fix($ \sc$)
the set of common fixed points of $\sc$, that is \\ Fix($ \sc$)=$\bigcap_{ s\in S}\{x\in C: T_{s}x=x  \}$.

%\begin{lem}{\rm(\cite{ma})}.\label{rho} Let $A $ be a strongly positive linear bounded operator on a Hilbert space $ H$ with coefficient $\overline{}$ and $0 < \rho\leq\|A\|^{-1}$ Then $\|I-\rho A\| \leq 1-\rho\,\overline{}$.\/\end{lem}

\begin{thm}{\rm(\cite{said})}.\label{tu}
Let $ S $ be a semigroup, let $ C $ be a closed, convex subset of a
reflexive Banach space $  E$,  $ \sc=\{T_{s}:s\in S\} $ be a representation of $S$ as nonexpansive mapping from $C$ into
itself such that weak closure of  $\lbrace T_{t}x : t \in S \rbrace $ is weakly compact for each $ x \in C $ and $X$ be a subspace of $B(S)$ such that $ 1 \in X $ and the mapping $t \rightarrow \left<T (t)x, x^{*}\right>$ be an element of $ X $ for
each $x \in C$ and $ x^{*} \in E$, and $\mu$ be a mean on $X$.
If we write $T_{\mu}x $ instead of
$\int\!T_{t}x\, d\mu(t),\/$
 then the following hold.\/\\{\rm({i})}\,  $T_{\mu}$ is a nonexpansive mapping from $C$ into $C$.\/\\{\rm({ii})} $T_{\mu}x=x $  for each  $x \in Fix(\sc) $.\/\\{\rm({iii})} $T_{\mu}x \in \overline{co}\,\{T_{t}x:t\in S\}$ for each $x\in C$.\/ \\{\rm({iv})} If $  X$ is $ r_{s} $-invariant  for
  each $ s \in S $ and  $ \mu $ is right invariant, then $T_{\mu}T_{t} =T_{\mu} $ for each $ t \in S $.
\end{thm}
\textbf{Remark}: From, Theorem 4.1.6 in \cite{tn}, every uniformly convex Banach space is strictly convex and reflexive.

 \begin{lem}\cite{ss}\label{sunyn}
 Let $S$ be a  semigroup and let $C$ be a compact convex subset of a real   strictly convex  and smooth Banach space   $E$.  Suppose that   $ \sc=\{T_{s}:s\in S\}$ be a representation of $S$ as nonexpansive mapping from $C$ into itself. Let $X$ be a left invariant subspace of $B(S)$ such that $1\in X$, and the function $t\mapsto \langle T_{t}x,x^{*}\rangle$ is an element of $X$ for each $x\in C$ and $x^{*}\in E^{*}$.  If $ \mu $  is a left invariant mean on  $ X $, then
 $\rm{ Fix}(T_{\mu})=T_{\mu}C=\rm{ Fix}(\sc) $  and there
exists a unique sunny nonexpansive retraction from $ C $ onto $\rm{ Fix}(\sc) $.
\end{lem}
Let $D$ be a subset of $B$ where $B$ is a subset of a Banach space $E$ and let $P$ be a retraction of $B$ onto $D$, that is,
$Px = x$ for each $x \in D$. Then $P$ is said to be sunny, if for each $x \in B$ and $t\geq0$ with $Px + t (x - Px) \in B$,
$P(Px + t (x - Px)) = Px$.
A subset $D$ of $B$ is said to be a sunny nonexpansive retract of $B$ if there exists a sunny nonexpansive retraction $P$ of
$B$ onto $D$.
We know that if $E$ is smooth and $P$ is a retraction of $B$ onto $D$, then $P$ is sunny and nonexpansive if and
only if for each $x \in B$ and $z \in D$,
$\langle  x-Px\,,\,J(z-Px)\rangle    \leq 0$.\\
For more details, see  \cite{tn}.

 Throughout the rest of this paper, the open ball of radius $r $ centered at 0 is denoted by $ B_{r}$. Let   $C$ be a nonempty closed  convex subset of a     Banach space  $E$. For $\epsilon > 0 $ and a mapping $ T : C \rightarrow C$, we let $ F_{\epsilon}(T)$ be the set of $\epsilon$-approximate fixed points of $T$, i.e. $ F_{\epsilon}(T) = \{x \in C : \|x - Tx\| \leq\epsilon\}$.

\section{Main result}

In this section, we deal with   a product topology convergence approximation scheme for finding an  element of  $E^{I}$  such that   it's  values are        the common fixed points  of      the  representations  of a  family of  the representations of   nonexpansive mappings.

\begin{thm}\label{g1}
Let $I$ be a nonempty set and  $S$ be a  semigroup. Let   $C_{i}$ be a nonempty compact convex subset of a real    strictly convex  and   reflexive  smooth Banach space $E$ for each $i \in I$.  Consider     the product space  $E^{I}$ with the  product topology generated by the strong topologies  on $E$ for each $i \in I$.  Suppose that   $ \sc_{i}=\{T_{s,i}:s\in S\}$ be a representation of $S$ as nonexpansive mapping from $C_{i}$ into itself such that $  \rm{Fix}(\sc_{i})\neq \emptyset$ for each $i \in I$.
 Let $X$ be a left invariant subspace of $B(S)$ such that $1\in X$, and the function $t\mapsto \langle T_{t}x,x^{*}\rangle$ is an element of $X$ for each $x\in C_{i}$ and $x^{*}\in E^{*}$. Let $\{\mu_{n}\}$ be a left regular sequence of  means on $X$.  Suppose that $f_{i}$ is an $\alpha_{i}$-contraction on $ C_{i}$ for each $  i \in I$. Let $\epsilon_{n}$ be a sequence in $(0, 1)$ such that $\displaystyle \lim_{n} \epsilon_{n}=0$.
  Then there exists a unique sunny nonexpansive retraction $ P_{i} $  of $ C_{i} $ onto $  \rm{Fix}( \sc_{i})$ and $ x_{i} \in C_{i} $ for each $  i \in I$ such that  the sequence  $\{g_{n}:I\rightarrow E \}$ in $E^{I}$       generated by
 \begin{equation}\label{4}
  \left\{
\begin{array}{lr}
g_{n}(i)= z_{n,i},\qquad i \in I, \\
      z_{n,i}=   \epsilon_{n} f_{i}(z_{n,i})+(1-\epsilon_{n})T_{\mu_{n,i}}z_{n,i} \qquad  i \in I,
\end{array} \right.
\end{equation}
      converges to the function  $g:I\rightarrow E$ defined by   $g(i)=P_{i}x_{i}$ in the product topology on $E^{I}$.
\end{thm}
\begin{proof}

        From Proposition 1.7.3 and Theorem 1.9.21 in \cite{Ag}, every   compact  subset $C_{i}$  of a  reflexive   Banach space  $E$, is weakly compact and   by Proposition 1.9.18 in \cite{Ag}, every closed convex subset of a weakly compact  subset $C_{i}$  of a      Banach space  $E$ is weakly compact  and  by Proposition 1.9.13 in \cite{Ag},     each convex subset $C_{i}$   of a normed space  $E$  is weakly closed if and only if $C_{i}$ is closed.
 Hence, weak closure of   $\lbrace T_{t,i}x : t \in S \rbrace $ is weakly compact for each $ x \in C_{i} $.

   The proof is divided  into six   steps.

  Step 1.  The existence of $z_{n,i}$ which satisfies \eqref{4}.\\Proof.
	This concludes   from the fact that the following  mapping $N_{n,i}$ is a contraction on $C_{i}$ for every $ n \in \mathbb{N}$ and $  i \in I$,
  \begin{align*}
 N_{n,i}x_{i}:=\epsilon_{n} f_{i}(x_{i})+(1-\epsilon_{n})  T_{\mu_{n,i}}x_{i}\quad(x_{i}\in C_{i}).
\end{align*}
Indeed, put  $\beta_{n}=(1+\epsilon_{n}(  \alpha_{i}-1) )$, then  $0 \leq\beta_{n} < 1 \; (n\in\mathbb{N})$. Hence, we have,
\begin{align*}
\|N_{n,i}x_{i}-N_{n,i}y_{i}\|  \leq &\epsilon_{n}   \|f_{i}(x_{i})-f_{i}(y_{i})\|+ \left(1-\epsilon_{n}\right)\|T_{\mu_{n,i}}x_{i}-T_{\mu_{n,i}}y_{i}\| \\  \leq & \epsilon_{n}  \alpha_{i} \|x_{i}-y_{i}\|+(1-\epsilon_{n})\|x_{i}-y_{i}\| \\ =&(1+\epsilon_{n}(  \alpha_{i}-1) )\|x_{i}-y_{i}\|=\beta_{n}\|x_{i}-y_{i}\|.
\end{align*}
Hence,  by Banach Contraction Principle (\cite{tn}), there exists a unique point $z_{n,i}\in C_{i}$   that $N_{n,i}z_{n,i}=z_{n,i}$.

  Step 2.  $\lim_{n\rightarrow \infty}\|z_{n,i}-T_{t,i}z_{n,i}\|=0 $,  for all  $  i \in I$ and  $t \in S.$\\
Proof.  Consider  $t \in S$,   $  i \in I$    and let $\epsilon>0  $.  By  Lemma 1 in \cite{shk1},  there exists $\delta>0  $ such that
 $  \overline{\text{co}}F_{\delta}(T_{t,i}) +2B_{\delta} \subseteq F_{\epsilon}(T_{t,i}) $. From   Corollary 2.8   in \cite{At2}, there
also exists a natural number $N $ such that
\begin{equation}\label{12}
    \Big\|\frac{1}{N+1}\sum_{j=0}^{N}T_{t^{j}s,i}y-T_{t,i}\Big(\frac{1}{N+1}\sum_{j=0}^{N}T_{t^{j}s,i}y \Big)\Big\| \leq \delta,
\end{equation}
for all $  s \in S $ and $y \in C_{i}$. Let $p_{i} \in \rm{Fix}(\sc_{i})$ and $M_{0,i}$ be a positive number such that, $\displaystyle\sup_{y \in C_{i}}\|y\|\leq M_{0,i}$. Let $t \in S$, from the fact that $\{\mu_{n}\} $ is strongly left regular, there exists $N_{0} \in \mathbb{N} $ such that
$\|\mu_{n}-l^{*}_{t^{j}}\mu_{n}\|\leq \frac{\delta}{(3M_{0,i})}$ for $n \geq N_{0}$ and $j = 1, 2, \cdots , N$. Therefore, we conclude
\begin{align}
\sup_{y \in C_{i}} \Big \|&T_{\mu_{n,i}}y-\int\! \frac{1}{N+1}\sum_{j=0}^{N}T_{t^{j}s,i}y\ud \mu_{n}(s) \Big  \|  \nonumber \\
=&\sup_{y\in C_{i}}\sup_{\|x^{*}\|=1} \Big | \langle T_{\mu_{n,i}}y ,x^{*}\rangle-\Big \langle \int\!\frac{1}{N+1}\sum_{j=0}^{N}T_{t^{j}s,i}y\ud \mu_{n}(s),x^{*}  \Big \rangle \Big  | \nonumber\\
   = &\sup_{y\in C_{i}}\sup_{\|x^{*}\|=1} \Big |\frac{1}{N+1}\sum_{i=0}^{N}(\mu_{n})_{s}\langle T_{s,i}y,x^{*}\rangle-\frac{1}{N+1}\sum_{j=0}^{N}(\mu_{n})_{s}\langle  T_{t^{j}s,i}y,x^{*}\rangle \Big  | \nonumber \\
   \leq & \frac{1}{N+1}\sum_{j=0}^{N}\sup_{y\in C_{i}}\sup_{\|x^{*}\|=1} \Big|(\mu_{n})_{s}\langle T_{s,i}y,x^{*}\rangle-(l^{*}_{t^{j}}\mu_{n})_{s}\langle T_{s,i}y,x^{*}\rangle \Big|\nonumber\\
   \leq &\max_{j=1,2,\cdots,N}\|\mu_{n}-l^{*}_{t^{j}}\mu_{n}\|(M_{0,i}+2\|p_{i}\|)\nonumber\\
   \leq &\max_{j=1,2,\cdots,N}\|\mu_{n}-l^{*}_{t^{j}}\mu_{n}\|(3M_{0,i})\nonumber \\ \leq&\delta\quad \rm{( n\geq N_{0})}.\label{13}
\end{align}
Applying  Theorem \ref{tu}, we conclude
\begin{equation}\label{14}
\int\!\frac{1}{N+1}\sum_{j=0}^{N}T_{t^{j}s,i}y\ud \mu_{n}(s)\in \overline{\text{co}}\left\{\frac{1}{N+1}\sum_{j=0}^{N}T_{t^{j},i}(T_{s,i}y):s\in S\right\}.
\end{equation}
we conclude from \eqref{12}-\eqref{14} that
\begin{align*}
 T_{\mu_{n,i}}y&\in \overline{\text{co}}\left\{\frac{1}{N+1}\sum_{i=0}^{N}T_{t^{j}s,i}y:s\in S\right\}+B_{\delta} \\
&\subset\overline{\text{co}}F_{\delta}(T_{t,i})+2B_{\delta}  \subset F_{\epsilon}(T_{t,i}),
\end{align*}
for all $y\in C_{i}$ and $n \geq N_{0}$. Hence,
$
\displaystyle \limsup_{n\rightarrow\infty}\sup_{y\in C_{i}}\|T_{t,i}(T_{\mu_{n,i}}y)-T_{\mu_{n,i}}y\|\leq \epsilon.
$
Because  $\epsilon > 0 $ is arbitrary, we get
\begin{align}\label{tun}
 \limsup_{n\rightarrow\infty}\sup_{y\in C_{i}}\|T_{t,i}(T_{\mu_{n,i}}y)-T_{\mu_{n,i}}y\|=0.
\end{align}
Let $t\in S $ and $\epsilon > 0$, then there exists $\delta > 0$, which satisfies \eqref{12}. Put $L_{0,i}=(1+\alpha_{i})2M_{0,i}+\| f_{i}(p_{i})-p_{i}\|$. Now, by the condition that  $\displaystyle \lim_{n} \epsilon_{n}=0$ and using  \eqref{tun}   there exists a natural number $N_{1}$ such that $T_{\mu_{n,i}}y \in F_{\delta}(T_{t,i})$ for each $ y  \in C_{i}$ and $\epsilon_{n}<\frac{\delta}{2L_{0,i}}$ for each $n \geq N_{1}$. Since  $ p_{i} \in \rm{Fix(\sc_{i}) }$, we conclude
\begin{align*}
\epsilon_{n}\|  f_{i}(z_{n,i})-& T_{\mu_{n,i}}z_{n,i}\|\\\leq & \epsilon_{n}\Big(\|  f_{i}(z_{n,i})-  f_{i}(p_{i})\|+\|  f_{i}(p_{i})- p_{i}\|\\& +\| T_{\mu_{n,i}}p_{i}- T_{\mu_{n,i}}z_{n,i}\|\Big)\\\leq &
 \epsilon_{n}\Big(  \alpha_{i}\|z_{n,i}-p_{i}\|+\|  f_{i}(p_{i})- p_{i}\|+\|z_{n,i}-p_{i}\|\Big)
\\\leq &\epsilon_{n}\left(  \alpha_{i}\|z_{n,i}-p_{i}\|+\|  f_{i}(p_{i})- p_{i}\|+\|z_{n,i}-p_{i}\|\right)\\
\leq & \epsilon_{n}\left((1+  \alpha_{i})\|z_{n,i}-p_{i}\|+\|  f_{i}(p_{i})- p_{i}\|\right)
 \\\leq &  \epsilon_{n}\left((1+  \alpha_{i})2M_{0,i}+\|  f_{i}(p_{i})- p_{i}\|\right)
\\=&\epsilon_{n}L_{0,i}\leq\frac{\delta}{2},
\end{align*}
for all $ n \geq N_{1}$. Observe that
\begin{align*}
z_{n,i}&=\epsilon_{n}  f_{i}(z_{n,i})+(1-\epsilon_{n})T_{\mu_{n,i}}z_{n,i}\\
&=T_{\mu_{n,i}}z_{n,i}+\epsilon_{n}\left(  f_{i}(z_{n,i})- T_{\mu_{n,i}}z_{n,i}\right)\\
&\in F_{\delta}(T_{t,i})+B_ {\frac{\delta}{2}} \\
&\subseteq F_{\delta}(T_{t,i})+2B_{\delta}\\
&\subseteq F_{\epsilon}(T_{t,i}).
\end{align*}
for each $ n \geq N_{1}$. Then we conclude that  \\
\indent \quad$\|z_{n,i}-T_{t,i}z_{n,i}\|\leq\epsilon\quad \rm{ (n \geq N_{1})}$.\\
Since $\epsilon > 0 $ is arbitrary, we conclude that
$\lim_{n\rightarrow \infty}\|z_{n,i}-T_{t,i}z_{n,i}\|=0 $.

 Step 3.  For each $i \in I$,  $\mathfrak{S}\{z_{n,i}\} \subset \rm{Fix(\sc_{i}) } $, where   $ \mathfrak{S}\{z_{n,i}\} $ denotes the set of strongly limit points of
 $ \{z_{n,i}\} $.\\
Proof. Let  $i \in I$. Consider
 $ z_{i}\in \mathfrak{S}\{z_{n,i}\} $ and let $\lbrace z_{n_{j},i}\rbrace $
 be a subsequence of $ \lbrace z_{n,i}\rbrace $ such
that $ z_{n_{j},i}\rightarrow z_{i} $.
\begin{align*}
   \Vert T_{t,i}z_{i}-z_{i}\Vert \leq&  \Vert T_{t,i}z_{i}-T_{t,i}z_{n_{j},i}\Vert + \Vert T_{t,i}z_{n_{j},i}-z_{n_{j}}\Vert+ \Vert z_{n_{j},i}-z_{i}\Vert \\ \leq & 2\Vert z_{n_{j},i}-z_{i}\Vert+ \Vert T_{t,i}z_{n_{j},i}-z_{n_{j},i}\Vert,
\end{align*}
applying step 2 we have,
\begin{align*}
   \Vert T_{t,i}z_{i}-z_{i}\Vert \leq 2\lim_{j}  \Vert z_{n_{j},i}-z_{i}\Vert+\lim_{j} \Vert T_{t,i}z_{n_{j},i}-z_{n_{j}}\Vert=0,
\end{align*}
hence,    $ z_{i}\in \rm{Fix(\sc_{i}) }$.

  Step 4.  For each $i \in I$, there exists a unique sunny nonexpansive retraction $ P_{i} $  of $ C _{i}$ onto $  \rm{Fix}(\sc_{i})$ and $ x_{i} \in C_{i} $ such that
\begin{align}\label{gbs}
  \Gamma_{i}  :=\limsup_{n}\langle x_{i}-P_{i}x_{i}\,,\,J(z_{n,i}-P_{i}x_{i})\rangle\leq0.
\end{align}
Proof. Applying  Lemma \ref{sunyn},   there exists a unique sunny nonexpansive retraction $ P_{i} $ of $ C_{i} $ onto
$ \rm{Fix(\sc_{i})} $. Using  Banach Contraction Mapping Principle, we have    that   $f_{i}P_{i}$ has a unique fixed point $x_{i}\in C_{i}$.  We prove  that
\begin{equation*}
  \Gamma_{i}  :=\limsup_{n}\langle x_{i}-P_{i}x_{i}\,,\,J(z_{n,i}-Px_{i})\rangle\leq0.
\end{equation*}

Observe  that, by the definition of
 $ \Gamma_{i} $  and by   the fact that  $C_{i}$ is a compact subset of $ E_{i} $,  we can select a subsequence $ \{z_{n_{j},i}\}$
 of $ \{z_{n,i}\}$ with the following properties:\\
(i)  $\displaystyle\lim_{j}\langle x_{i}-P_{i}x_{i}\,,\,J(z_{n_{j},i}-P_{i}x_{i})\rangle=\Gamma_{i}$;\\
(ii) $\{z_{n_{j},i}\}$ converges strongly to a point $z_{i}$;\\
applying  Step 3, we have $ z_{i} \in  \rm{Fix}(\sc_{i}) $.  From the fact that  $E_{i}  $  is smooth, we conclude
\begin{align*}
\Gamma_{i} =& \lim_{j}\langle  x_{i}-P_{i}x_{i}\,,\,J(z_{n_{j},i}-P_{i}x_{i})\rangle=\langle  x_{i}-P_{i}x_{i}\,,\,J(z_{i}-P_{i}x_{i})\rangle    \leq 0.
\end{align*}
 Since     $ f_{i}P_{i}x_{i}=x_{i}$, we have  $ ( f_{i}-I)P_{i}x_{i}=x_{i}-P_{i}x_{i} $. From    page 99  in \cite{tn}, we have, for each $n\in \mathbb{N}  $,
\begin{align*}
\epsilon_{n}  (   \alpha_{i} &-1)   \Vert  z_{n,i} -P_{i}x_{i} \big \Vert ^{2} \\ \geq &
\Big [ \epsilon_{n}   \alpha_{i} \big \Vert  z_{n,i} -P_{i}x_{i} \big \Vert+(1-\epsilon_{n}) \Vert z_{n,i}-P_{i}x_{i} \Vert \Big]^{2}-\Vert z_{n,i}-P_{i}x_{i} \Vert^{2} \\  \geq &
\Big [ \epsilon_{n}   \big \Vert  f(z_{n,i}) -f(P_{i}x_{i}) \big \Vert+(1-\epsilon_{n}) \Vert T_{\mu_{n,i}}z_{n,i}-P_{i}x_{i} \Vert \Big]^{2}-\Vert z_{n,i}-P_{i}x_{i} \Vert^{2} \\  \geq &
2\Big \langle \epsilon_{n}  \Big ( f(z_{n,i})-f(P_{i}x_{i})  \Big)\\ &+ (1-\epsilon_{n})(T_{\mu_{n,i}}z_{n,i}-P_{i}x_{i})-(z_{n,i}-P_{i}x_{i})\, , \,J(z_{n,i}-P_{i}x_{i}) \Big \rangle \\ = & -2\epsilon_{n} \langle  (  f-I)P_{i}x_{i}\, , \,J(z_{n,i}-P_{i}x_{i})\rangle  \\ = & -2\epsilon_{n} \langle x_{i}-P_{i}x_{i}\, , \,J(z_{n,i}-P_{i}x_{i})\rangle ,
\end{align*}
therefore,
\begin{align}\label{gbh}
 \Vert z_{n,i}-P_{i}x_{i} \Vert^{2}\leq \frac{2}{1-   \alpha_{i} }\langle x_{i}-P_{i}x_{i}\, , \,J(z_{n,i}-P_{i}x_{i})\rangle.
\end{align}

Step 5. $ \{z_{n,i}\}$  strongly converges to $P_{i}x_{i}$.\\
Proof.   \eqref{gbs}, \eqref{gbh} and  the fact that   $ P_{i}x_{i} \in  \rm{Fix}(\sc_{i})$, imply that
\begin{align*}
    \limsup_{n}\|z_{n,i}-P_{i}x_{i}\|^{2}\leq &  \frac{2}{1-   \alpha_{i}}\limsup_{n}\langle x_{i}-P_{i}x_{i} \,,\,J(z_{n,i}-P_{i}x_{i})\rangle\leq0.
\end{align*}
Hence, $z_{n,i} \rightarrow P_{i}x_{i}$.

 Step 6.    $\{g_{n}  \}$  converges to    $g$  in the product topology on $E^{I}$. \\
 Proof. As we know,  the topology on $E^{I}$ is that of  strongly pointwise convergence. Hence from \eqref{4} and  step 5, we conclude the results.
\end{proof}
\section{Examples and Corollaries}
In this section, we deal with some examples and corollaries.
\begin{co}
Let $I$ be a nonempty set. Let   $C_{i}$ be a nonempty compact convex subset of a real   Hilbert space $H$ for each $i \in I$.  Consider     the product space  $H^{I}$ with the  product topology generated by the strong topologies  on $H$.  Suppose that   $\{T_{i}\}_{i \in I}$ is a family of nonexpansive mappings from $C_{i}$ into itself such that $  \rm{Fix}(T_{i})\neq \emptyset$ for each $i \in I$.
   Suppose that $f_{i}$ is an $\alpha_{i}$-contraction on $ C_{i}$ for each $  i \in I$. Let $\epsilon_{n}$ be a sequence in $(0, 1)$ such that $\displaystyle \lim_{n} \epsilon_{n}=0$.
  Then there exists a unique sunny nonexpansive retraction $ P_{i} $  of $ C_{i} $ onto $  \rm{Fix}(T_{i})$ and $ x_{i} \in C_{i} $ for each $  i \in I$ such that  the sequence  $\{g_{n}:I\rightarrow H \}$ in $H^{I}$       generated by
 \begin{equation*}
  \left\{
\begin{array}{lr}
g_{n}(i)= z_{n,i},\qquad i \in I, \\
      z_{n,i}=   \epsilon_{n} f_{i}(z_{n,i})+(1-\epsilon_{n})\frac{1}{n}\sum_{k=1}^{n}T^{k}_{i}z_{n,i} \qquad  i \in I,
\end{array} \right.
\end{equation*}
      converges to the function  $g:I\rightarrow E$ defined by   $g(i)=P_{i}x_{i}$ in the product topology on $E^{I}$.
\end{co}
\begin{proof}
Let  $ \sc_{i}=\{T^{j}_{i}:j\in S\}  $  where  $ S = \{1,2, . . .\} $.  For a function $f=(x_{1}, x_{2}, ...) \in B(S)$, define
 \begin{equation*}
    \mu_{n}(f)=\frac{1}{n}\sum_{k=1}^{n}x_{k} \qquad( n \in \mathbb{N}),
\end{equation*}
then   $\{\mu_{n}\}$ is a left regular sequence of  means on $B(S)$  \cite{tn}. Hence, we have
\begin{equation*}
   T_{\mu_{n,i}}x=\frac{1}{n}\sum_{k=1}^{n}T^{k}_{i}x \qquad( n \in \mathbb{N}).
\end{equation*}
Then from Theorem \ref{g1}, we get the results.
\end{proof}

\begin{co}
Let $I$ be a nonempty set. Let   $C_{i}$ be a nonempty compact convex subset of a real   Hilbert space $H$ for each $i \in I$.  Consider     the product space  $H^{I}$ with the  product topology generated by the strong topologies  on $H$. Suppose that $ S = \mathbb{R}^{+}=\{t \in \mathbb{R}:0\leq t <+\infty\} $ and   $ \sc_{i}=\{T_{t,i}:t\in \mathbb{R}^{+}\}  $ be a representation of $\mathbb{R}^{+}$ as nonexpansive mapping from $C_{i}$ into itself such that $  \rm{Fix}(\sc_{i})\neq \emptyset$ for each $i \in I$. Let $X$ be a left invariant subspace of $C(\mathbb{R}^{+})$ such that $1\in X$.
   Suppose that $f_{i}$ is an $\alpha_{i}$-contraction on $ C_{i}$ for each $  i \in I$. Let $\{\epsilon_{n}\}\subseteq(0, 1)$ such that $\displaystyle \lim_{n} \epsilon_{n}=0$ and $\{ a_{n}\}\subseteq (0,\infty)$ such that  $\displaystyle\lim_{n}a_{n}= \infty$.
  Then there exists a unique sunny nonexpansive retraction $ P_{i} $  of $ C_{i} $ onto $  \rm{Fix}(T_{i})$ and $ x_{i} \in C_{i} $ for each $  i \in I$ such that  the sequence  $\{g_{n}:I\rightarrow H \}$ in $H^{I}$       generated by
 \begin{equation*}
  \left\{
\begin{array}{lr}
g_{n}(i)= z_{n,i},\qquad i \in I, \\
      z_{n,i}=   \epsilon_{n} f_{i}(z_{n,i})+(1-\epsilon_{n})\frac{1}{a_{n}}\int_{0}^{a_{n}}\!T_{t,i}z_{n,i}\ud t\quad\qquad ( n \in \mathbb{N}), \qquad  i \in I,
\end{array} \right.
\end{equation*}
      converges to the function  $g:I\rightarrow E$ defined by   $g(i)=P_{i}x_{i}$ in the product topology on $E^{I}$.
\end{co}
\begin{proof}
  For a function $f\in C(\mathbb{R}^{+})$, define
 \begin{equation*}
    \mu_{n}(f)=\frac{1}{a_{n}}\int_{0}^{a_{n}}\!f(t)\ud t \qquad( n \in \mathbb{N}),
\end{equation*}
then   $\{\mu_{n}\}$ is a left regular sequence of  means on $B(S)$  \cite{tn}. Hence, we have
\begin{equation*}
   T_{\mu_{n,i}}x=\frac{1}{a_{n}}\int_{0}^{a_{n}}\!T_{t, i}x \ud t \qquad( n \in \mathbb{N}).
\end{equation*}
Then from Theorem \ref{g1}, we get the results.
\end{proof}
\begin{center}
 %  \bf{ Acknowledgements}
\end{center}
%The  author is  grateful to  the University of Lorestan for their support.

\end{large}
% ------------------------------------------------------------------------
\end{document}